\documentclass[11pt]{article}
\usepackage{amsmath,amsthm,paralist}
\usepackage[footnotesize]{caption}

\usepackage{url}
\usepackage{pxfonts}
\usepackage{hyperref}
\usepackage{enumerate}
\usepackage{nicefrac}
\usepackage{algpseudocode}
\usepackage{algorithm}

\usepackage[margin=1in]{geometry}

\newcommand{\bea}{\begin{array}}
\newcommand{\eea}{\end{array}}
\newcommand{\beq}{\begin{equation}}
\newcommand{\eeq}{\end{equation}}

\newcommand{\sh}{\mathop{\text{sh}}}
\renewcommand{\sc}{\mathop{\text{sc}}}

\newtheorem{Th}{Theorem}

\newtheorem{Lemma}[Th]{Lemma}

\title{On partitions into squares of distinct integers whose reciprocals sum to 1}
\author{Max A. Alekseyev}

\date{The George Washington University\\
{\tt maxal@gwu.edu}}

\begin{document}
\maketitle\thispagestyle{plain}

Integers $11$ and $24$ share an interesting property: each can be partitioned into distinct positive integers whose reciprocals sum to $1$.
Indeed, $11=2+3+6$ where $\nicefrac12 + \nicefrac13 + \nicefrac16 = 1$, and $24=2+4+6+12$ where again $\nicefrac12 + \nicefrac14 + \nicefrac16 + \nicefrac{1}{12} = 1$.
Sums of reciprocals of distinct positive integers are often referred to as \emph{Egyptian fractions}~\cite{Knott2017}, and from known Egyptian fractions of $1$ we can easily construct other numbers with the same property.
The smallest such numbers are $1$, $11$, $24$, $30$ (sequence \texttt{A052428} in the OEIS~\cite{OEIS}) and they tend to appear rather sparsely among small integers.
So it may come as a surprise that \textit{any} number greater than $77$ has this property. This was proved in 1963 by Graham~\cite{Graham1963}, who further 
conjectured that any sufficiently large integer can be partitioned into \textit{squares} of distinct positive integers whose reciprocals sum to 1~\cite[Section D11]{Guy2013UPINT}.
Examples of such partitions again can be obtained from known examples of Egyptian fractions of $1$:
\begin{center}
\begin{tabular}{lll}
$1 = \frac11$ & ~~~ & $1^2 = 1$ \\
\\
$1 = \frac12 + \frac13 + \frac16$ & ~~~ & $2^2 + 3^2 + 6^2 = 49$ \\
\\
$1 = \frac12 + \frac14 + \frac16 + \frac1{12}$ & ~~~ & $2^2 + 4^2 + 6^2 + 12^2 = 200$ \\
\\
$1 = \frac12 + \frac13 + \frac1{10} + \frac1{15}$ & ~~~ & $2^2 + 3^2 + 10^2 + 15^2 = 338$ \\
\end{tabular}
\end{center}
In fact, $1,\ 49,\ 200,\ 338$ are the smallest such numbers (sequence \texttt{A297895} in the OEIS~\cite{OEIS}) and they seemingly appear even more sparsely.
Nevertheless, in the present study we prove Graham's conjecture and establish the exact bound for existence of such partitions.

We call a positive integer $m$ \emph{representable} if there exists a set of positive integers $X=\{x_1,x_2,\dots,x_k\}$ such that
$$
1 = \frac{1}{x_1} + \dots + \frac{1}{x_n}\quad\mbox{and}\quad m=x_1^2 + \dots + x_n^2.
$$
We further say that $X$ is a \emph{representation} of $m$. For example, $200$ is representable as it has representation $\{2,4,6,12\}$.

Our main result is outlined in the following theorem.

\begin{Th}\label{th:main} 
The largest not representable integer is $8542$.
\end{Th}

We provide a proof of Theorem~\ref{th:main} generalizing the original approach of Graham~\cite{Graham1963} based on constructing representations of larger numbers from those of smaller ones.
More generally, we refer to such construction as \emph{translation} of representations, and introduce a class of translations that acts on restricted representations.
This provides us with yet another proof of Theorem~\ref{th:main}.

Since our approach requires computation of representations of certain small numbers, and we start with discussion of an algorithm that generates representations of a given number.
The same algorithm is also used to prove that $8542$ is not representable.

\section{Computing Representations}

Our goal is to design an efficient exhaustive-search algorithm for a representation of a given integer $m$. We start with proving bounds for the search, using the power mean inequality.

Recall that the $q$-th power mean of positive numbers $x_1,\dots,x_k$ is defined as 
$$A_q(x_1,\dots,x_k) = \left(\frac{x_1^q+\dots+x_k^q}{k}\right)^{\frac{1}{q}}.$$
In particular, $A_q(x_1,\dots,x_k)$ represents the harmonic, geometric, and arithmetic mean when $q=-1,\,0,\,1$, respectively.\footnote{Formally speaking, the geometric mean equals the limit of $A_q(x_1,\dots,x_k)$ as $q\to 0$.}
The power mean inequlity, generalizing the AM-GM inequlity, states that
$A_q(x_1,\dots,x_k) \leq A_{q'}(x_1,\dots,x_k)$ whenever $q\leq q'$.

The following lemma will be crucial for our algorithm.

\begin{Lemma}\label{lem:ineq}
For a positive integer $d$ and a finite set of positive integers $X$, let
\begin{equation}\label{eq:sm}
s = \sum_{x\in X} \frac{1}{x}\qquad\text{and}\qquad n = \sum_{x\in X} x^d.
\end{equation}
Then
\begin{equation}\label{eq:boundk}
|X| \leq s\sqrt[d+1]{\frac{n}{s}}
\end{equation}
and
\begin{equation}\label{eq:rangex1}
\left\lceil\frac{1}{s}\right\rceil \leq \min X \leq \left\lfloor \min\left\{ \sqrt[d+1]{\frac{n}{s}},\ \sqrt[d]{n} \right\} \right\rfloor.
\end{equation}
\end{Lemma}

\begin{proof} Suppose that $X=\{x_1,\dots,x_k\}$, where $x_1 < x_2 < \dots < x_k$, and so $|X|=k$ and $\min X=x_1$.

From \eqref{eq:sm} and the power mean inequality, it follows that
$$\frac{k}{s} = A_{-1}(x_1,\dots,x_k) \leq A_d(x_1,\dots,x_k) = \sqrt[d]{\frac{n}{k}},$$
which further implies \eqref{eq:boundk}.

Since $x_1 = \min X$, we have
$$\frac{1}{x_1}\leq s = \frac{1}{x_1} + \dots + \frac{1}{x_k} \leq \frac{k}{x_1},$$
implying that
$$\frac{1}{s}\leq x_1 \leq \frac{k}{s}.$$
Similarly, from $x_1^d \leq x_1^d + \dots + x_k^d = n$, we obtain $x_1 \leq \sqrt[d]{n}$.
Finally, using \eqref{eq:boundk}, we get
$$x_1 \leq \frac{k}{s} \leq \sqrt[d+1]{\frac{n}{s}},$$
which completes the proof of \eqref{eq:rangex1}.
\end{proof}

Lemma~\ref{lem:ineq} for $d=2$ enables us to search for a representation $X=\{x_1 < x_2 < \dots < x_k\}$ of a given integer $m$ using backtracking as follows. 
Clearly, $X$ should satisfy the equalities \eqref{eq:sm} for $s=1$ and $n=m$, and so we let $s_1=1$ and $n_1=m$. Then the value of $x_1=\min X$ lays in the range given by \eqref{eq:rangex1} for $s=s_1$ and $n=n_1$.
For each candidate value of $x_1$ in this range, we compute $s_2=s_1-\frac{1}{x_1}$ and $n_2=n_1-x_1^2$ representing the sum of reciprocals and squares, respectively, of the elements of $X\setminus\{x_1\}$.
Then \eqref{eq:rangex1} for $s=s_2$ and $n=n_2$ defines a range for $x_2=\min (X\setminus\{x_1\})$ (additionally we require $x_2\geq x_1+1$), and so on. 
The procedure stops when $s_{k+1}=0$ and $n_{k+1}=0$ for some $k$, implying that $1=s_1=\sum_{i=1}^k \frac{1}{x_i}$ and $m = n_1 = \sum_{i=1}^k x_i^2$, i.e., $\{x_1,\dots,x_k\}$ is a representation of $m$.
On the other hand, if all candidate values have been explored without finding a representation, then no such representation exists (i.e., $m$ is not representable).

We remark that the bounds in \eqref{eq:rangex1} do not depend on $|X|$, and thus we do not need to know the size of $X$ in advance. 
Furthermore, the inequality \eqref{eq:boundk} guarantees that the algorithm always terminates, and either produces a representation of $m$ or establishes that none exist.

\begin{algorithm}[!t]
\caption{For given rational number $s$ and integers $t,\,n$, function $\textsc{ConstructX}(t,\ s,\ n)$ constructs a set $X$ 
of positive integers such that $\min X\geq t$, $\sum_{x\in X} x^{-1}=s$, and $\sum_{x\in X} x^2=n$; or returns the empty set $\emptyset$ if no such $X$ exists.}
\label{alg:findX}
\begin{algorithmic}
\Function{ConstructX}{$t,\ s,\ n$}
   \If {$s\leq 0$ or $n\leq 0$}
   \State \Return $\emptyset$
   \EndIf
\State $L :=  \left\lceil \max\left\{\ \frac{1}{s},\ t\ \right\}\right\rceil$
\State $U := \left\lfloor\min\left\{\ \sqrt[3]{\frac{n}{s}},\ \sqrt{n}\ \right\}\right\rfloor$
\For {$x := L,\ L+1,\ \dots,\ U$}
   \State {$s_{\text{new}} := s - \frac{1}{x}$}
   \State {$n_{\text{new}} := n - x^2$}
   \If {$s_{\text{new}}=0$ and $n_{\text{new}}=0$}
   \State \Return $\{x\}$
   \EndIf
   \State {$X :=$ \Call{ConstructX}{$x+1,\ s_{\text{new}},\ n_{\text{new}}$}}
   \If {$X \ne \emptyset$}
   \State \Return {$X\cup \{x\}$} 
   \EndIf
\EndFor
\State \Return $\emptyset$
\EndFunction
\end{algorithmic}
\end{algorithm}

Algorithm~\ref{alg:findX} presents a pseudocode of the described algorithm as the recursive function $\textsc{ConstructX}(t,s,n)$.
To construct a representation of $m$, one needs to call $\textsc{ConstructX}(1,1,m)$. For $m=8542$, this function returns the empty set and thus implies the following statement.

\begin{Lemma}\label{lem:un8542} 
The number $8542$ is not representable.
\end{Lemma}

It is easy to modify Algorithm~\ref{alg:findX} to search for representations with additional restrictions on the elements (e.g., with certain numbers forbidden). We will see a need for such representations below.

\section{Proof of Theorem~\ref{th:main}}

For a set $S$, we say that a representation $X$ is \emph{$S$-avoiding} if $X\cap S=\emptyset$, i.e., $X$ contains no elements from $S$.

Graham~\cite{Graham1963} introduced two functions: 
\begin{equation}\label{eq:f03}
f_0(X) = \{2\} \cup 2X\qquad\text{and}\qquad f_3(X) = \{ 3,7,78,91\} \cup 2X
\end{equation}
defined on $\{39\}$-avoiding representations, where the set $2X$ is obtained from $X$ by multiplying each element by 2 (i.e., $2X = \{ 2x\ :\ x\in X\}$).
Indeed, since a representation of $m>1$ cannot contain $1$, the sets $\{2\}$ and $2X$ are disjoint, implying that
$$\sum_{y\in f_0(X)} \frac{1}{y} = \frac{1}{2} + \sum_{x\in X} \frac{1}{2x} = \frac{1}{2} + \frac{1}{2} \sum_{x\in X} \frac{1}{x} = 1.$$
Similarly, the sets $\{ 3,7,78,91\}$ and $2X$ are disjoint, since $2X$ consists of even numbers and $2\cdot 39=78\notin 2X$ (as $X$ is $\{39\}$-avoiding), implying that
$$\sum_{y\in f_3(X)} \frac{1}{y} = \frac13 + \frac17 + \frac{1}{78} + \frac{1}{91} + \sum_{x\in X} \frac{1}{2x} = \frac{1}{2} + \frac{1}{2} \sum_{x\in X} \frac{1}{x} = 1.$$
Furthermore, if $X$ is a representation of an integer $m$, then $\sum_{y\in 2X} y^2 = \sum_{x\in X} (2x)^2 = 4m$, implying that $f_0(X)$ and $f_3(X)$ 
are representations of integers $g_0(m)=2^2 + 4m = 4m+4$ and $g_3(m)=3^2 + 7^2 + 78^2 + 91^2 + 4m = 4m + 14423$, respectively.
Trivially, for any integer $m$, we have $g_0(m)\equiv 0\pmod{4}$ and $g_3(m)\equiv 3\pmod{4}$, which explain the choice of indices in the function names.
Lastly, one can easily check that neither of $f_0(X)$ and $f_3(X)$ contains $39$, and thus they map $\{39\}$-avoiding representations to $\{39\}$-avoiding representations.

While functions $f_0(X)$ and $f_3(X)$ were sufficient for the problem addressed by Graham, we will need two more functions $f_1(X)$ and $f_2(X)$ such that the corresponding functions $g_i(m)$ ($i=0,1,2,3$) form a complete residue system modulo $4$.
It turns out that such functions cannot be defined on $\{39\}$-avoiding representations, which leads us to further restricting the domain. 
Namely, we find it convenient to deal with the $\{21,39\}$-avoiding representations, on which we define
\begin{equation}\label{eq:f12}
f_1(X) = \{ 5,7,9,45,78,91\} \cup 2X\qquad\text{and}\qquad f_2(X) = \{3,7,42\} \cup 2X.
\end{equation}
One can easily see that the functions $f_1$ and $f_2$ map a $\{21,39\}$-avoiding representation $X$ of an integer $m$ to a $\{21,39\}$-representation of integers $g_1(m)=4m+16545$ and $g_2(m)=4m+1822$, respectively.
The functions $f_0$ and $f_3$ can also be viewed as mappings on $\{21,39\}$-avoiding representations.
As planned, we have $g_i(m)\equiv i\pmod{4}$ for each $i=0,1,2,3$, which will play a key role in our proof of Theorem~\ref{th:main} below.

\begin{Lemma}\label{lem:no2139} Let $m$ be an integer such that $8543\leq m\leq 54533$. Then
\begin{enumerate}[(i)]
\item $m$ is representable;
\item $m$ has a $\{21,39\}$-avoiding representation unless 
$m\in E$, where
$$E = \{\ 8552,\ 8697,\ 8774,\ 8823,\ 8897,\ 8942,\ 9258,\ 9381,\ 9439,\ 9497\ \}.$$
\end{enumerate}
\end{Lemma}

\begin{proof} The proof is established computationally.
We provide representations of 
all representable integers $m\leq 54533$ in a supplementary file (see Section~\ref{sec:supp}), where the listed representation of each $m\geq 8543$, $m\notin E$ is $\{21,39\}$-avoiding.
\end{proof}

We are now ready to prove Theorem~\ref{th:main}.

\paragraph{Proof of Theorem~\ref{th:main}.}
Thanks to Lemma~\ref{lem:un8542}, it remains to prove that every number greater than $8542$ is representable.
First, Lemma~\ref{lem:no2139} implies that every number $m$ in the range $8543\leq m\leq 9497$ is representable. For larger $m$, we will use induction on $m$ to prove that 
every integer $m\geq 9498$ has a $\{21,39\}$-avoiding representation.

Again, by Lemma~\ref{lem:no2139}, we have that every number $m$ in the range $9498\leq m\leq 54533$ has a $\{21,39\}$-avoiding representation.

Consider $m>54533$ and assume that all integers in the interval $[9498,m-1]$ have $\{21,39\}$-avoiding representations.
Let $i=m\bmod 4$ and so $i\in\{0,1,2,3\}$. Then there exists an integer $m'$ such that $m=g_i(m')$ and thus $m' \geq \frac{m-16545}{4} > \frac{54533-16545}{4} = 9497$. By the induction assumption, $m'$ has a $\{21,39\}$-avoiding representation $X$. 
Then $f_i(X)$ is a $\{21,39\}$-avoiding representation of $m$, which concludes the proof.
\qed

\section{$t$-Translations}

The functions defined in \eqref{eq:f03}-\eqref{eq:f12} inspire us to consider a more broad class of functions that map representations of small integers to those of larger ones.
For an integer $t$, we call an integer $m$ \emph{$t$-representable} if it has a representation $X$ with $\min X\geq t$, called $t$-representation.
Clearly, an integer $m>1$ is representable if and only if it is $2$-representable.
A $t$-representation (if it exists) of a given integer $m$ can be constructed with Algorithm~\ref{alg:findX} by calling $\textsc{ConstructX}(t,1,m)$.

Let $t$ be a positive integer. A tuple of positive integers $r=(k;y_1,\dots,y_l)$ is called \emph{$t$-translation}
if
$$
1 - \frac{1}{k} = \frac{1}{y_1} + \dots + \frac{1}{y_l},
$$
$t\leq y_1 < y_2 < \dots < y_l$, and for every $i\in\{1,2,\dots,l\}$, we have either $y_i<tk$ or $k\nmid y_i$.

With every $t$-translation $r=(k;y_1,\dots,y_l)$ we associate two parameters:
\emph{scale} $\sc(r)=k^2$ and \emph{shift} $\sh(r)=y_1^2 + \dots + y_l^2$.

\begin{Lemma}\label{Lrep} Let $m$ be a $t$-representable integer, and $r$ be a $t$-translation.
Then the number $\sc(r)\cdot m + \sh(r)$ is $t$-representable.
\end{Lemma}

\begin{proof} Suppose that $\{x_1,\dots,x_n\}$ is a $t$-representation of $m$.
If $r=(k;y_1,\dots,y_l)$ is a $t$-translation, then
$$
\sc(r)\cdot m + \sh(r) = y_1^2 + \dots + y_l^2 + (k x_1)^2 + \dots + (k x_n)^2
$$
and
$$
\frac{1}{y_1} + \dots + \frac{1}{y_l} + \frac{1}{k x_1} + \dots + \frac{1}{k x_n} 
= \left(1-\frac{1}{k}\right) + \frac{1}{k} = 1.
$$
Notice that $\{ y_1, \dots, y_l\} \cap \{k x_1, \dots, k x_n\}=\emptyset$,
since for any $i\in\{1,2,\dots,l\}$ and $j\in\{1,2,\dots,n\}$, we have either $y_i < tk \leq k x_j$ or $k\nmid y_i$, i.e., $y_i\ne kx_j$.
Hence, the set $Y = \{ y_1, \dots, y_l, k x_1, \dots, k x_n \}$ forms a representation of $\sc(r)\cdot m + \sh(r)$.
Furthermore, it is easy too that $\min Y\geq t$, i.e., $Y$ is a $t$-representation.
\end{proof}

The function $f_0(X)$ defined in \eqref{eq:f03} corresponds to a $2$-translation $(2;2)$, however $(2;3,7,78,91)$ corresponding to the function $f_3(X)$ is not a $2$-translation because of the presence of $78=2\cdot 39$.
As we will see below, it is preferable to have the scale small, ideally equal $2^2=4$, which is possible for the only $2$-translation $(2;2)$.

At the same time, for $t=6$, we can construct a set of $6$-translations such as
\begin{equation}\label{eq:transl6}
\begin{split}
\{~~ & (2;9, 10, 11, 15, 21, 33, 45, 55, 77),\quad (2;6, 7, 9, 21, 45, 105),\\
& (2;7, 9, 10, 15, 21, 45, 105);\quad (2;6, 9, 11, 21, 33, 45, 55, 77) ~~\},
\end{split}
\end{equation}
where the translations have scale $2^2=4$ and shifts $\{ 13036, 13657, 13946, 12747 \}$.

A set of translations $S$ is called \emph{complete} if the set
$$\left\{\ \sh(r) \pmod{\sc(r)}\ :\ r\in S\ \right\}$$
forms a complete residue system, i.e., for any integer $m$ there exists $r\in S$ such that $m\equiv \sh(r) \pmod{\sc(r)}$.
It can be easily verified that 
\eqref{eq:transl6} forms a complete set of translations.

The following theorem generalizes Theorem~\ref{th:main}.

\begin{Th}\label{th:Tcomp} Let $S$ be a complete set of $t$-translations with
maximum scale $q$ and maximum shift $s$.
If numbers $n+1,n+2,\dots,qn+s$ are $t$-representable, then so is any number greater than $n$.
\end{Th}

\begin{proof} 
Suppose that all numbers $n+1,n+2,\dots,qn+s$ are $t$-representable.
We will prove by induction that so is any number $m>qn+s$. 

Assume that all numbers from $n+1$ to $m-1$ are $t$-representable.
Since $S$ is complete, there exists a $t$-translation $r\in S$
such that $\sh(r)\equiv m\pmod{\sc(r)}$. Note that $\sh(r)\leq s$ and $\sc(r)\leq q$.

For a number $m'=\frac{m-\sh(r)}{\sc(r)}$, we have $m'<m$ and $m'>\frac{qn+s - s}{q} = n$, implying by induction that $m'$ is $t$-representable. 
Then by Lemma~\ref{Lrep} the number $m=\sc(r)\cdot m'+\sh(r)$ is $t$-representable.
\end{proof}

We can observe that small 6-representable numbers tend to appear very sparsely, with the smallest ones being $2579$, $3633$, $3735$, $3868$ (sequence \texttt{A303400} in the OEIS~\cite{OEIS}).
Nevertheless, using the complete set of $6$-translations \eqref{eq:transl6}, we can prove the following theorem.

\begin{Th}\label{th:rep6} The largest integer that is not $6$-representable is $15707$.
\end{Th}

\begin{proof} First, we computationally establish that $15707$ is not $6$-representable.
Thanks to the complete set of $6$-translations \eqref{eq:transl6}, 
by Theorem~\ref{th:Tcomp}, it remains to show that all integers $m$ in the interval $15708\leq m\leq 4\cdot 15707+13946=76774$ are $6$-representable, which we again establish computationally (see Section~\ref{sec:supp}).
\end{proof}

We have also computed similar bounds for other $t\leq 8$ (sequence \texttt{A297896} in the OEIS~\cite{OEIS}), although we do not know whether such bounds exist for all $t$.

Theorem~\ref{th:rep6} together with Lemma~\ref{lem:no2139}(i) provides yet another proof of Theorem~\ref{th:main}. 

\section{Supplementary Files}\label{sec:supp}

We provide the following supplementary files to support our study:
\begin{itemize}
\item \url{https://oeis.org/A297895/a297895.txt} contains representations of all representable integers $m\leq 54533$, which are $\{21,39\}$-avoiding if $m\geq 8543$ and $m\notin E$ (see Lemma~\ref{lem:no2139});
\item \url{https://oeis.org/A303400/a303400.txt} contains $6$-representations of all $6$-representable integers up to $76774$, which include all integers in the interval $[15708,76774]$.
\end{itemize}

\bibliographystyle{plain}
\bibliography{eg.bib}

\end{document}